\documentclass{amsart}
\title{Lie elements in the group algebra}
\author[Yu.M.Burman]{Yurii M.\ Burman}
\address{Independent University of Moscow (119002, 11, B.Vlassievsky per., 
Moscow, Russia) and Higher School of Economics (National Research 
University; 101000, 20, Myasnitskaya str., Moscow, Russia)}
\email{burman@mccme.ru}
\date{}

\makeatletter
\RequirePackage{amssymb}

\newcommand{\theoremName}{Theorem}
\newcommand{\lemmaName}{Lemma}
\newcommand{\corollaryName}{Corollary}
\newcommand{\statementName}{Proposition}

\newcommand{\remarkName}{Remark}
\newcommand{\exampleName}{Example}
\newcommand{\definitionName}{Definition}
\newcommand{\problemName}{Problem}
\newcommand{\proofName}{Proof}
\renewcommand{\proofname}{\proofName}
\newcommand{\answerName}{Answer}
\newcommand{\hintName}{Hint}

\theoremstyle{plain}

\newtheorem {proposition}{\statementName}
\newtheorem {Theorem}{\theoremName}

\newtheorem {Corollary}{\corollaryName}

\theoremstyle{remark}

\theoremstyle{definition}

\newtheorem{Definition}{\definitionName}

\let\@newpf\proof 
\let\proof\relax

\def \namepf[#1] {\@newpf[\proofname\ #1]}
\newenvironment{proof}{\@ifnextchar[{\namepf}{\@newpf[\proofname]}}{\qed\endtrivlist}

\newcounter{qst}
\@addtoreset{qst}{problem}

\def \Integer {{\mathbb Z}}

\def \Complex {{\mathbb C}}

\def \lnorm#1\rnorm {\vphantom{#1}\left\|\smash{#1}\right\|}
\def \lmod#1\rmod {\vphantom{#1}\left|\smash{#1}\right|}

\newcommand \bydef {\stackrel{\mbox{\scriptsize def}}{=}}

\@ifundefined{name}{\newcommand \name[1] {\mathop{\rm #1}\nolimits}}%
{\relax}

\renewcommand \phi {\varphi}
\renewcommand \rho {\varrho}

\makeatother

\begin{document}

 \begin{abstract}
Given a representation $V$ of a group $G$, there are two natural ways of 
defining a representation of the group algebra $k[G]$ in the external power 
$V^{\wedge m}$. The set $\mathcal L(V)$ of elements of $k[G]$ for which 
these two ways give the same result is a Lie algebra and a representation 
of $G$. For the case when $G$ is a symmetric group and $V = \Complex^n$, a 
permutation representation, these spaces $\mathcal L(\Complex^n)$ are 
naturally embedded into one another. We describe $\mathcal L(\Complex^n)$ 
for small $n$ and formulate questions and conjectures for future research.
 \end{abstract}

\maketitle
\def \thesubsubsection {\arabic{subsubsection}}

\section{Setting and motivation}

Let $V$ be a finite-dimensional representation of a group $G$ over a field 
$k$. For every $g \in G$ and every $m$ define linear operators $A_m(g), 
B_m(g): V^{\wedge m} \to V^{\wedge m}$ as follows: 
 \begin{align*}
A_m(g) (v_1 \wedge \dots \wedge v_m) &= g(v_1) \wedge \dots 
\wedge g(v_m)\\
B_m(g) (v_1 \wedge \dots \wedge v_m) &= \sum_{p=1}^m v_1
\wedge \dots \wedge g(v_p) \wedge \dots \wedge v_m.
 \end{align*}
(here and below $v_1, \dots, v_m$ are arbitrary vectors in $V$). Then 
extend the operators $A_m, B_m: G \to \name{End}(V^{\wedge m})$ to the 
group algebra $k[G]$ by linearity. Also take by definition $A_0(g) = 1$ (an 
operator $k \to k$) and $B_0(g) = 0$ for every $g \in G$. 

 \begin{Definition}
An element $x \in k[G]$ satisfying $A_m(g) = B_m(g)$ for all $m = 0,1, 
\dots, \dim V$ is called a Lie element of $k[G]$ (with respect to the 
representation $V$). The set of Lie elements is denoted by $\mathcal L(V)$.
 \end{Definition}

Besides the associative algebra structure in $k[G]$ and 
$\name{End}(V^{\wedge m})$ consider an associated Lie algebra structure in 
them, taking $[p,q] = pq-qp$. 

 \begin{proposition}\label{Pp:Lie} 
Maps $A_m, B_m: k[G] \to \name{End}(V^{\wedge m})$ are Lie algebra 
homomorphisms. 
 \end{proposition}

 \begin{proof}
It is clear that $A_m: G \to \name{End}(V^{\wedge m})$ is an associative 
algebra homomorphism ($A_m(x y) = A_m(x) A_m (y)$ for all $x, y \in G$), 
hence a Lie algebra homomorphism. For $B_m$ take $x = \sum_{g \in G} 
a_g g$, $y = \sum_{h \in G} b_h h$, to obtain
 \begin{align*}
B_m(x) &B_m(y) v_1 \wedge \dots \wedge v_m = \sum_{h \in 
G, 1 \le p \le m} b_h B_m(x) v_1 \wedge \dots \wedge 
h(v_p) \wedge \dots \wedge v_m\\
&= \sum_{g,h \in G, 1 \le p \le m} a_g b_h v_1 \wedge \dots \wedge 
g(h(v_p)) \wedge \dots \wedge v_m \\
&\hphantom{=}+ \sum_{g,h \in G, 1 \le p,q \le m, p \ne 
q} a_g b_h v_1 \wedge \dots \wedge h(v_p) \wedge \dots \wedge g(v_q) \wedge 
\dots \wedge v_m,
 \end{align*}
whence $B_m([x,y]) = [B_m(x), B_m(y)]$.
 \end{proof}

 \begin{Corollary}
The set of Lie elements $\mathcal L(V) \subset k[G]$ is a Lie subalgebra.
 \end{Corollary}

 \begin{proposition} \label{Pp:Conj}
For any $x, y \in k[G]$ and any $m$ one has $y A_m(x) y^{-1} = 
€_m(yxy^{-1})$ and $y B_m(x) y^{-1} = B_m(yxy^{-1})$.
 \end{proposition}

The proof is evident.

 \begin{Corollary}
The set $\mathcal L(V) \subset k[G]$ is a representation of $G$ where 
elements of the group act by conjugation.
 \end{Corollary}

This note takes its origin from the paper \cite{BurmanZvonkine}. The paper 
contains a formula for the so called Hurwitz generating function which 
lists factorizations of a cyclic permutation $(12 \dots n)$ to a product of 
transpositions. The key ingredient of the proof of the formula is the fact 
that $1-(ij) \in \mathcal L(\Complex^n)$ where $\Complex^n$ is the 
permutation representation of the symmetric group (see Proposition 
\ref{Pp:Transp} below). Any other element $x \in \mathcal L(\Complex^n)$ 
corresponds to a generalization of this result producing a formula listing 
factorizations of the cycle to a product of various permutations with 
various weights; the weights depend on $x$. Equivalently, the same formula 
lists graphs embedded into oriented surfaces so that their complement is 
homeomorphic to a disk; any $x \in \mathcal L(\Complex^n)$ generates a 
formula listing similar embeddings of multi-graphs (again, with the weights 
depending on $x$).

This note is a description of research in progress; see the list of 
questions and conjectures at the end.

\section{The symmetric group case}

Here we take $G = S_n$, $n = 2, 3, \dots$. Let $k = \Complex$ and $V$ be an 
$n$-dimensional permutation representation of $S_n$ (the group acts on 
elements of the basis $x_1, \dots, x_n \in \Complex^n$ permuting their 
indices). We'll be writing $\mathcal L_n$ for short, instead of $\mathcal 
L(\Complex^n)$.

 \begin{proposition}[cf.\ \cite{BurmanZvonkine}]\label{Pp:Transp}
$1 - (ij) \in \mathcal L_n$ for all $1 \le i < j \le n$.
 \end{proposition}

 \begin{proof}
Take any $v_1, \dots, v_m \in V$; then $A_m(1) v_1 \wedge \dots \wedge v_m 
= v_1 \wedge \dots \wedge v_m$ and $B_m(1) v_1 \wedge \dots \wedge v_m = m 
v_1 \wedge \dots \wedge v_m$.

It follows from Proposition \ref{Pp:Conj} that without loss of generality 
one may assume $i = 1, j = 2$. Apparently, this is enough to take for $v_s$ 
basic vectors: $v_s = x_{i_s}$ for all $s = 1, \dots, m$, where $1 \le i_1 
< \dots < i_m \le n$ are any indices. Consider now three cases:

\subsubsection{$i_1, \dots, i_m \ne 1, 2$} Then 
 \begin{align*}
&A_m((12)) (x_{i_1} \wedge \dots \wedge x_{i_m}) = x_{i_1} \wedge \dots 
\wedge x_{i_m}, \\
&B_m((12)) (x_{i_1} \wedge \dots \wedge x_{i_m}) = m x_{i_1} \wedge \dots 
\wedge x_{i_m},
 \end{align*}
so that 
 \begin{equation*}
A_m(1-(12)) (x_{i_1} \wedge \dots \wedge x_{i_m}) = 0 = B_m(1-(12)) 
(x_{i_1} \wedge \dots \wedge x_{i_m}).
 \end{equation*}

\subsubsection{$i_1 = 1$, $i_2, \dots, i_m \ne 1,2$} Then 
 \begin{align*}
&A_m((12)) (x_1 \wedge x_{i_2} \wedge \dots \wedge x_{i_m}) = x_2 \wedge 
x_{i_2} \wedge \dots \wedge x_{i_m},\\
&B_m((12)) (x_2 + (m-1)x_1) \wedge x_{i_2} \wedge \dots \wedge x_{i_m}) = 
x_2 \wedge x_{i_2} \wedge \dots \wedge x_{i_m},
 \end{align*}
so that 
 \begin{align*}
A_m(1-(12)) (x_1 \wedge x_{i_2} \wedge \dots \wedge x_{i_m}) &= (x_1-x_2) 
\wedge x_{i_2} \wedge \dots \wedge x_{i_m}\\
&= B_m(1-(12)) (x_1 \wedge x_{i_2} \wedge \dots \wedge x_{i_m})
 \end{align*}

\subsubsection{$i_1 = 1, i_2 = 2$} Then 
 \begin{align*}
&A_m((12)) (x_1 \wedge x_2 \wedge x_{i_3} \wedge \dots \wedge x_{i_m}) = - 
x_1 \wedge x_2 \wedge x_{i_3} \wedge \dots \wedge x_{i_m},\\
&B_m((12)) (x_1 \wedge x_2 \wedge x_{i_3} \wedge \dots \wedge x_{i_m}) = 
(m-2) x_1 \wedge x_2 \wedge x_{i_3} \wedge \dots \wedge x_{i_m},
 \end{align*}
so that 
 \begin{align*}
A_m(1-(12)) (x_1 \wedge x_2 \wedge x_{i_3} \wedge \dots \wedge x_{i_m}) &= 
2 x_1 \wedge x_2 \wedge x_{i_3} \wedge \dots \wedge x_{i_m}\\
&= B_m(1-(12)) (x_1 \wedge x_2 \wedge x_{i_3} \wedge \dots \wedge x_{i_m}).
 \end{align*}
 \end{proof}

Denote by $\iota_n: S_n \to S_{n+1}$ a standard embedding: for any 
permutation $\sigma \in S_n$ take $\iota_n(\sigma)(k) = \sigma(k)$ for any 
$1 \le k \le n$ and $\iota_n(\sigma)(n+1) = n+1$. The embedding can be 
extended by linearity to an algebra homomorphism $\iota_n: \Complex[S_n] 
\to \Complex[S_{n+1}]$.

 \begin{proposition}
$\iota_n(\mathcal L_n) \subset \mathcal L_{n+1}$.
 \end{proposition}

 \begin{proof}
Let $u = \sum_{\sigma \in S_n} a_\sigma \sigma \in \mathcal L_n$. Like in 
Proposition \ref{Pp:Transp} above, it is enough to consider the action of 
$\iota_n(u)$ on $x \bydef x_{i_1} \wedge \dots \wedge x_{i_m}$ where $1 \le 
i_1 < \dots < i_m \le n+1$. Consider two cases.

\setcounter{subsubsection}{0}
\subsubsection{$i_m \le n$} Then $A_m(\iota_n(u))(x) = 
A_m(u)(x) = B_m(u)(x) = B_m(\iota_n(u))(x)$, so that $\iota_n(u) \in 
\mathcal L_{n+1}$.

\subsubsection{$i_m = n+1$} Then $A_m(\iota_n(u))(x) = 
A_{m-1}(u)(x_{i_1} \wedge \dots \wedge x_{i_{m-1}}) \wedge x_{n+1}$. On the 
other hand,
 \begin{equation*}
B_m(\iota_n(u))(x) = \left(\sum_{\sigma \in S_n} a_\sigma \sum_{p=1}^n 
x_{i_1} \wedge \dots \wedge x_{\sigma(i_p)} \wedge \dots \wedge 
x_{i_{m-1}}\right) \wedge x_{n+1} + \sum_{\sigma \in S_n} a_\sigma \cdot x.
 \end{equation*}
One has $A_0(u) = \sum_{\sigma \in S_n} a_\sigma$ and $B_0(u) = 0$. Once $u 
\in \mathcal L_n$, the last term in the equation above is zero, so
 \begin{equation*}
B_m(\iota_n(u))(x) = B_{m-1}(u)(x_{i_1} \wedge \dots \wedge x_{i_{m-1}}) 
\wedge x_{n+1},
 \end{equation*}
whence $A_m(\iota_n(u))(x) = B_m(\iota_n(u))(x)$, and again $\iota_n(u) \in 
\mathcal L_{n+1}$.
 \end{proof}

\section{$\mathcal L_n$ for small $n$}

One has $\dim \mathcal L_2 = 1$. The space is spanned by $1 - (12) \in 
\Complex[S_2]$, is a trivial Lie algebra and a trivial representation of 
$S_2 = \Integer/2\Integer$. 

The space $\mathcal L_3$ contains elements $1-(12)$, $1-(23)$ and $1-(13)$ 
by Proposition \ref{Pp:Transp}. By the corollary of Proposition 
\ref{Pp:Lie} it also contains $[1-(12),1-(23)] = (123)-(132)$ (by $(i_1 
\dots i_k) \in S_n$ we mean a cyclic permutation sending every $i_s$ to 
$i_{s+1 \bmod k}$). Easy calculations show that these elements form a basis 
in $\mathcal L_3$, so that $\dim \mathcal L_3 = 4$. The space $\mathcal 
L_3$ splits, as a representation of $S_3$, to the trivial representation 
$V_0$ (spanned by $1 - (12)/3 - (13)/3 - (23)/3$), sign representation 
$V_1$ (spanned by $(123) - (132)$) and a two-dimensional representation 
$V_2$ (spanned by $(12)-(13)$, $(13)-(23)$ and $(23)-(12)$; the elements 
sum up to zero, and any two of them form a basis). As a Lie algebra 
$\mathcal L_3$ is a direct sum of the center $V_0$ and a three-dimensional 
subalgebra spanned by $V_1 \cup V_2$. (This statement is partly true for 
any $n$: $\mathcal L_n$ contains a trivial representation, which lies in 
its center as a Lie algebra.)

The space $\mathcal L_4$ contains, by Proposition \ref{Pp:Transp}, the $6$ 
elements $1-(ij)$, $1 \le i < j \le 4$. By Propositions \ref{Pp:Lie} and 
\ref{Pp:Conj} it also contains all the elements $(ijk) - (ikj) = 
[1-(ij),(1-(jk)]$, $1 \le i < j < k \le 4$ (totally $4$), and the elements 
$\gamma_1 = [1-(14),(123)-(132)] = (1234)+(1432)-(1243)-(1342)$ and 
$\gamma_2 = [1-(24), (123)-(132)] = (1243)+(1342)-(1324)-(1423)$. Easy 
computer-assisted computations show that these $12$ elements form a basis 
in $\mathcal L_4$. 

As a representation of $S_4$, $\mathcal L_4$ contains a $6$-dimensional 
representation spanned by $1-(ij)$, $1 \le i < j \le 4$; it splits 
into a trivial representation spanned by $1 - \frac16 \sum_{1 \le i < j \le 
4} (ij)$, a $3$-dimensional representation of the type $(3,1)$ and a 
$2$-dimensional representation of the type $(2,2)$. Another $4$-dimensional 
subrepresentation of $\mathcal L_4$ is spanned by $(ijk) - (ikj)$, $1 \le i 
< j < k \le 4$; it splits into a sign representation (spanned by $\sum_{1 
\le i < j < k \le 4} (ijk)-(ikj)$) and a $3$-dimensional representation of 
the type $(2,1,1)$. The elements $\gamma_1$ and $\gamma_2$ span a 
$2$-dimensional subrepresentation. Totally, $\mathcal L_4$ contains a 
trivial representation, a sign representation, two copies of a 
$2$-dimensional representation and two nonisomorphic $3$-dimensional 
representations.

\section{Questions and conjectures}

\subsection{Dimension and representations} For an arbitrary $n$, what is 
the dimension of $\mathcal L_n$? A refinement of the question: find the 
Frobenius character $R_n = \sum_{\lmod \lambda\rmod = n} a_\lambda 
\chi_\lambda$ of the representation $\mathcal L_n$; here the sum runs over 
all partitions of $n$, $a_\lambda$ is the multiplicity in $\mathcal L_n$ of 
the irreducible representation of $S_n$ of the type $\lambda$, and 
$\chi_\lambda$ is the Schur polynomial corresponding to $\lambda$. 

\subsection{Generators} 
\def \theoremName {Conjecture}

 \begin{Theorem}
The Lie algebra $\mathcal L_n$ is generated by the elements $\nu_{ij} = 1 - 
(ij)$, $1 \le i < j \le n$.
 \end{Theorem}
Computations confirm the conjecture for $n \le 5$.

\subsection{Action on the original representation} The elements of 
$\mathcal L(V) \subset k[G]$ act in the original representation $V$ of the 
group $G$. This action may have a kernel. These kernels and quotients of 
$\mathcal L(V)$ by them sometimes exhibit interesting properties:

 \begin{Theorem}
Let $K_n$ be a kernel of the action of $\mathcal L_n$ in the permutation 
representation $\Complex^n$. Then $\dim \mathcal L_n/K_n = (n-1)!$. The 
repeated commutators 
 \begin{equation*}
[\dots[\nu_{1i_1},\nu_{2i_2}],\nu_{3i_3}],\dots],\nu_{n-1,i_{n-1}}]
 \end{equation*}
for all $i_1, \dots, i_{n-1}$ such that $s+1 \le i_s \le n$ for all $s = 1, 
\dots, n-1$ form a basis in $\dim \mathcal L_n/K_n$.
 \end{Theorem}

\section*{Grants and acknowledgements}

The final stage of the work was supported by the RFBR grant NSh-5138.2014.1 
``Singularities theory and its applications'', by the Higher School of 
Economics (HSE) Scientific foundation grant 12-01-0015 ``Differential 
geometry on graphs and discrete path integration'', by the by Dobrushin 
professorship grant 2013 (the Independent University of Moscow) and by the 
Simons foundation grant (autumn 2013).

\end{document}